\documentclass[a4paper,11pt]{article}
\usepackage{amsmath,amsthm,amssymb,enumitem,amscd,calc,titlesec}
\usepackage[bookmarks=false]{hyperref}
\usepackage{comment}

\renewcommand{\thesection}{\arabic{section}}
\titleformat{\section}{\Large\bf\boldmath}{\thesection.}{2ex}{}{}
\titlespacing{\section}{0ex}{2ex}{1ex}

\renewcommand{\thesubsection}{\arabic{section}.\arabic{subsection}}
\titleformat{\subsection}{\large\bf\boldmath}{\thesubsection.}{2ex}{}{}
\titlespacing{\section}{0ex}{1.5ex}{0.5ex}

\numberwithin{equation}{section}
{\theoremstyle{definition}\newtheorem{definition}{Definition}[section]

\newtheorem{remark}[definition]{Remark}
\newtheorem{remarkletter}{Remark}
}

\newtheorem{lemma}[definition]{Lemma}
\newtheorem{theorem}[definition]{Theorem}

\newtheorem{theoremletter}[remarkletter]{Theorem}
\newtheorem{propositionletter}[remarkletter]{Proposition}
\newtheorem{corollaryletter}[remarkletter]{Corollary}

\newcommand{\M}{\operatorname{M}}
\newcommand{\C}{\mathbb{C}}
\newcommand{\F}{\mathbb{F}}
\newcommand{\cR}{\mathcal{R}}
\newcommand{\actson}{\curvearrowright}
\newcommand{\SL}{\operatorname{SL}}

\newcommand{\Z}{\mathbb{Z}}

\newcommand{\id}{\mathord{\operatorname{id}}}

\newcommand{\cE}{\mathcal{E}}
\newcommand{\recht}{\rightarrow}
\newcommand{\cU}{\mathcal{U}}
\newcommand{\vphi}{\varphi}

\newcommand{\R}{\mathbb{R}}
\newcommand{\al}{\alpha}
\newcommand{\eps}{\varepsilon}
\newcommand{\Tr}{\operatorname{Tr}}
\newcommand{\ovt}{\mathbin{\overline{\otimes}}}

\newcommand{\om}{\omega}

\newcommand{\cZ}{\mathcal{Z}}

\newcommand{\cK}{\mathcal{K}}

\newcommand{\cH}{\mathcal{H}}

\newcommand{\ot}{\otimes}
\newcommand{\cL}{\mathcal{L}}

\newcommand{\cG}{\mathcal{G}}
\newcommand{\cM}{\mathcal{M}}
\newcommand{\dpr}{^{\prime\prime}}
\newcommand{\be}{\beta}

\newcommand{\Mtil}{\widetilde{M}}

\newcommand{\cN}{\mathcal{N}}

\newcommand{\cS}{\mathcal{S}}

\newcommand{\Om}{\Omega}
\newcommand{\supp}{\operatorname{supp}}

\newcommand{\HNN}{\operatorname{HNN}}
\newcommand{\cC}{\mathcal{C}}

\newcommand{\cI}{\mathcal{I}}
\newcommand{\cT}{\mathcal{T}}
\newcommand{\closure}[1]{\operatorname{closure}(#1)}

\begin{document}

\begin{center}
{\boldmath\LARGE\bf  Normalizers inside amalgamated\vspace{0.5ex}\\ free product von Neumann algebras}
\bigskip

{\sc by Stefaan Vaes\footnote{KU~Leuven, Department of Mathematics, Leuven (Belgium), stefaan.vaes@wis.kuleuven.be \\
    Supported by ERC Starting Grant VNALG-200749, Research
    Programme G.0639.11 of the Research Foundation~--
    Flanders (FWO) and KU~Leuven BOF research grant OT/08/032.}}
\end{center}

\medskip

\begin{abstract}\noindent
Recently, Adrian Ioana proved that all crossed products $L^\infty(X) \rtimes (\Gamma_1*\Gamma_2)$ by free ergodic probability measure preserving actions of a nontrivial free product group $\Gamma_1 * \Gamma_2$ have a unique Cartan subalgebra up to unitary conjugacy. Ioana deduced this result from a more general dichotomy theorem on the normalizer $\cN_M(A)\dpr$ of an amenable subalgebra $A$ of an amalgamated free product von Neumann algebra $M = M_1 *_B M_2$. We improve this dichotomy theorem by removing the spectral gap assumptions and obtain in particular a simpler proof for the uniqueness of the Cartan subalgebra in $L^\infty(X) \rtimes (\Gamma_1*\Gamma_2)$.
\end{abstract}

\section{Introduction and main results}

All free ergodic nonsingular group actions $\Gamma \actson (X,\mu)$ on standard probability spaces give rise to a crossed product von Neumann algebra $L^\infty(X) \rtimes \Gamma$, in which $L^\infty(X)$ is a \emph{Cartan subalgebra.} More generally, Cartan subalgebras arise as $L^\infty(X) \subset L(\cR)$ where $\cR$ is a countable nonsingular Borel equivalence relation on $(X,\mu)$. One of the main questions in the classification of these von Neumann algebras $L^\infty(X) \rtimes \Gamma$ and $L(\cR)$ is whether or not $L^\infty(X)$ is their unique Cartan subalgebra up to unitary conjugacy. Indeed, if uniqueness holds, the classification problem is reduced to classifying the underlying (orbit) equivalence relations.

Within Popa's deformation/rigidity theory, there has been a lot of recent progress on the uniqueness of Cartan subalgebras in II$_1$ factors, starting with \cite{OP07} where it was shown that all crossed products $L^\infty(X) \rtimes \F_n$ by free ergodic probability measure preserving (pmp) \emph{profinite} actions of the free groups $\F_n$ have a unique Cartan subalgebra. Note that this provided the first uniqueness theorem for Cartan subalgebras up to unitary conjugacy. The result of \cite{OP07} was gradually extended to profinite actions of larger classes of groups $\Gamma$ in \cite{OP08,CS11,CSU11}, but all relied on profiniteness of the action and weak amenability of the group $\Gamma$. At the same time, it was conjectured that crossed products $L^\infty(X) \rtimes \F_n$ by actions of the free groups could have a unique Cartan subalgebra without any profiniteness assumptions on $\F_n \actson (X,\mu)$.

In a joint work with Popa in \cite{PV11,PV12}, we solved this conjecture and proved that the free groups $\Gamma=\F_n$ and all nonelementary hyperbolic groups $\Gamma$ are $\cC$-rigid (Cartan-rigid), i.e.\ for every free ergodic pmp action $\Gamma \actson (X,\mu)$, the II$_1$ factor $L^\infty(X) \rtimes \Gamma$ has a unique Cartan subalgebra up to unitary conjugacy. We obtained this result as a consequence of a general dichotomy theorem about normalizers of amenable subalgebras in crossed product von Neumann algebras $N\rtimes \Gamma$, arising from trace preserving actions of such groups $\Gamma$ on arbitrary tracial $(N, \tau)$.

Then in \cite{Io12}, the general dichotomy result of \cite{PV11} has been exploited to establish $\cC$-rigidity for arbitrary nontrivial free products $\Gamma=\Gamma_1 * \Gamma_2$ and large classes of amalgamated free products $\Gamma = \Gamma_1 *_\Sigma \Gamma_2$. This provided in particular the first non weakly amenable $\cC$-rigid groups. The main idea of \cite{Io12} is to use the free malleable deformation from \cite{IPP05} of a crossed product $B \rtimes (\Gamma_1 * \Gamma_2)$, providing a $1$-parameter family of embeddings $\theta_t : B \rtimes (\Gamma_1 * \Gamma_2) \recht N \rtimes \F_2$ into some crossed product by the free group $\F_2$. Then the main result of \cite{PV11} is applied to this crossed product $N \rtimes \F_2$ and a very careful and delicate analysis is needed to ``come back'' and deduce results about the original crossed product $B \rtimes (\Gamma_1 * \Gamma_2)$.

The purpose of this article is to give a simpler approach to this ``come back'' procedure and, at the same, prove a more general result removing the spectral gap assumptions in \cite{Io12}. As a result, we obtain a simpler proof for the $\cC$-rigidity of amalgamated free product groups.

Our method allows to prove a more generic theorem about the normalizer of a subalgebra inside an amalgamated free product of von Neumann algebras, see Theorem \ref{thm.AFP} below. This theorem has the advantage to immediately imply a similar result for HNN extensions of von Neumann algebras, see Theorem \ref{thm.HNN}. As such we obtain with the same effort $\cC$-rigidity for a large class of HNN extensions $\Gamma = \HNN(\Gamma_1,\Sigma,\theta)$, established before in \cite{DI12} using more involved methods.

As we explain below, following the strategy of \cite{HV12}, we also prove a uniqueness theorem of Cartan subalgebras in type III factors. This then allows us to give the first examples of type III actions $\Gamma \actson (X,\mu)$ that are W$^*$-superrigid, i.e.\ such that the group $\Gamma$ and its action $\Gamma \actson (X,\mu)$ can be recovered from $L^\infty(X) \rtimes \Gamma$, up to induction of actions.

To state the main result of the article, we first recall Popa's theory of \emph{intertwining-by-bimodules} from \cite{Po01,Po03}. When $(M,\tau)$ is a tracial von Neumann algebra and $A \subset pMp$, $B \subset M$ are von Neumann subalgebras, we say that $A \prec_M B$ ($A$ embeds into $B$ inside $M$) if $L^2(pM)$ admits a non-zero $A$-$B$-subbimodule that is finitely generated as a right Hilbert $B$-module. This is ``almost'' equivalent with the existence of a partial isometry $v \in B$ such that $v A v^* \subset B$. By \cite[Theorem 2.1 and Corollary 2.3]{Po03}, the negation $A \not\prec_M B$ is equivalent with the existence of a net of unitaries $(a_i)_{i \in I}$ in $\cU(A)$ satisfying $\lim_i \|E_B(xu_i y)\|_2 = 0$ for all $x,y \in M$.

Also recall from \cite[Definition 2.2]{OP07} that $A$ is said to be \emph{amenable relative to $B$ inside $M$} if there exists an $A$-central state $\Omega$ on Jones' basic construction von Neumann algebra $p \langle M,e_B \rangle p$ satisfying $\Om(x) = \tau(x)$ for all $x \in pMp$. When $B$ is amenable, this is equivalent to $A$ being amenable. When $M = D \rtimes \Gamma$ and $\Lambda, \Sigma < \Gamma$ are subgroups, then the relative amenability of $D \rtimes \Lambda$ w.r.t.\ $D \rtimes \Sigma$ is equivalent with the relative amenability of $\Lambda$ w.r.t.\ $\Sigma$ inside $\Gamma$, i.e.\ with the existence of a $\Lambda$-invariant mean on $\Gamma/\Sigma$.

The following is the main result of the article. The same result was proven in \cite[Theorem 1.6]{Io12} under the additional assumption that the normalizer $\cN_{pMp}(A) = \{u \in \cU(pMp) \mid u A u^* = A \}$ of $A$ inside $pMp$ has spectral gap.

\begin{theoremletter}\label{thm.AFP}
Let $M = M_1 *_B M_2$ be the amalgamated free product of the tracial von Neumann algebras $(M_i,\tau)$ with common von Neumann subalgebra $B \subset M_i$ w.r.t.\ the unique trace preserving conditional expectations. Let $p \in M$ be a nonzero projection and $A \subset pMp$ a von Neumann subalgebra that is amenable relative to one of the $M_i$ inside $M$. Then at least one of the following statements holds.
\begin{itemize}
\item $A \prec_M B$.
\item There is an $i \in \{1,2\}$ such that $\cN_{pMp}(A)\dpr \prec_M M_i$.
\item We have that $\cN_{pMp}(A)\dpr$ is amenable relative to $B$ inside $M$.
\end{itemize}
\end{theoremletter}

As in \cite{Io12}, several uniqueness theorems for Cartan subalgebras can be deduced from Theorem \ref{thm.AFP}. This is in particular the case for II$_1$ factors $M = L(\cR)$ that arise from a countable pmp equivalence relation $\cR$ that can be decomposed as a free product $\cR = \cR_1 * \cR_2$ of subequivalence relations $\cR_i \subset \cR$. Since we now no longer need to prove the spectral gap assumption, we can directly deduce from Theorem \ref{thm.AFP} the following improvement of \cite[Corollary 1.4]{Io12} and \cite[Theorem 6.3]{BHR12}.

\begin{corollaryletter}\label{cor.product-equiv}
Let $\cR$ be a countable ergodic pmp equivalence relation on the standard probability space $(X,\mu)$. Assume that $\cR = \cR_1 * \cR_2$ for two subequivalence relations $\cR_i \subset \cR$. Assume that $|\cR_1 \cdot x| \geq 3$ and $|\cR_2 \cdot x| \geq 2$ for a.e.\ $x \in X$. Then $L^\infty(X)$ is the unique Cartan subalgebra of $L(\cR)$ up to unitary conjugacy.
\end{corollaryletter}

A tracial von Neumann algebra $(M,\tau)$ is called \emph{strongly solid} if for every diffuse amenable von Neumann subalgebra $A \subset M$, the normalizer $\cN_M(A)\dpr$ is still amenable. For completeness, we also show how to deduce from Theorem \ref{thm.AFP} the following stability result for strong solidity under amalgamated free products, slightly improving on \cite[Theorem 1.8]{Io12}.

For the formulation of the result, recall from \cite[Section 3]{Po03} that an inclusion $B \subset (M_1,\tau)$ of tracial von Neumann algebras is called \emph{mixing} if for every sequence $b_n \in B$ with $\|b_n\| \leq 1$ for all $n$ and $b_n \recht 0$ weakly, we have that $\lim_n \|E_B(xb_ny)\|_2 = 0$ for all $x,y \in M_1 \ominus B$. Typical examples of mixing inclusions arise as $L(\Sigma) \subset L(\Gamma)$ when $\Sigma < \Gamma$ is a subgroup such that $g \Sigma g^{-1} \cap \Sigma$ is finite for all $g \in \Gamma - \Sigma$, or as $L(\Sigma) \subset B \rtimes \Sigma$ whenever $\Sigma$ acts in a mixing and trace preserving way on $(B,\tau)$.

\begin{corollaryletter}\label{cor.strong-solid}
Let $(M_i,\tau_i)$ be strongly solid von Neumann algebras with a common amenable von Neumann subalgebra $B \subset M_i$ satisfying ${\tau_1}_{|B} = {\tau_2}_{|B}$. Assume that the inclusion $B \subset M_1$ is mixing. Denote by $M = M_1 *_B M_2$ the amalgamated free product w.r.t.\ the unique trace preserving conditional expectations. Then $M$ is strongly solid.
\end{corollaryletter}

On the level of tracial von Neumann algebras, by \cite{Ue07}, amalgamated free products and HNN extensions are one and the same thing, up to amplifications. Therefore, Theorem \ref{thm.AFP} has an immediate counterpart for HNN extensions that we formulate as Theorem \ref{thm.HNN} below.

As a consequence, we can then reprove \cite[Theorem 1.1]{Io12} and \cite[Corollary 1.7]{DI12}, showing $\cC$-rigidity for amalgamated free product groups, HNN extensions and their direct products. We refer to Theorem \ref{thm.Crigid} for a precise statement.

Finally in Section \ref{sec.Wstar-superrigid}, we use the methods of \cite{HV12} to deduce from Theorem \ref{thm.AFP} a uniqueness theorem for Cartan subalgebras in type III factors $L^\infty(X) \rtimes \Gamma$ arising from nonsingular free ergodic actions of amalgamated free product groups, see Theorem \ref{thm.unique-cartan-type-III}, generalizing \cite[Theorem D]{BHR12}. As a consequence, we can provide the following first nonsingular actions \emph{of type III} that are W$^*$-superrigid.

\begin{propositionletter}\label{prop.Wstar-superrigid}
Consider the linear action of $\SL(5,\Z)$ on $\R^5$ and define the subgroup $\Sigma < \SL(5,\Z)$ of matrices $A$ satisfying $A e_i = e_i$ for $i=1,2$. Put $\Gamma = \SL(5,\Z) *_\Sigma (\Sigma \times \Z)$ and denote by $\pi : \Gamma \recht \SL(5,\Z)$ the natural quotient homomorphism. The diagonal action $\Gamma \actson \R^5/\R_+ \times [0,1]^\Gamma$ given by
$$g \cdot (x,y) = (\pi(g) \cdot x, g \cdot y) \; ,$$
where $g \cdot y$ is given by the Bernoulli shift, is a nonsingular free ergodic action of type III$_1$ that is W$^*$-superrigid.

This means that for any nonsingular free action $\Gamma' \actson (X',\mu')$, the following two statements are equivalent.
\begin{itemize}
\item $L^\infty(X) \rtimes \Gamma \cong L^\infty(X') \rtimes \Gamma'$.
\item There exists an embedding of $\Gamma$ into $\Gamma'$ such that $\Gamma' \actson X'$ is conjugate with the induction of $\Gamma \actson X$ to a $\Gamma'$-action.
\end{itemize}
\end{propositionletter}

To clarify the statement of Proposition \ref{prop.Wstar-superrigid}, one should make the following observations. Contrary to the case of probability measure preserving actions, it not relevant to consider stable isomorphisms, since the type III factor $M = L^\infty(X) \rtimes \Gamma$ is isomorphic with $B(H) \ovt M$ for every separable Hilbert space $H$. For the same reason, it is unavoidable that $\Gamma' \actson (X',\mu')$ can be any induction of $\Gamma \actson (X,\mu)$ and need not be conjugate to $\Gamma \actson (X,\mu)$ itself.

It is also possible to prove that for $0 < \lambda < 1$, the analogous action of $\Gamma$ on $\R^5/\lambda^\Z \times [0,1]^\Gamma$ is of type III$_\lambda$ and W$^*$-superrigid in the appropriate sense. The correct formulation is necessarily more intricate because the action is by construction orbit equivalent with the action of $\Gamma \times \Z$ on $\R^5 \times [0,1]^\Gamma$. More generally for a type III$_\lambda$ free ergodic action $\Gamma \actson (X,\mu)$, there always is a canonically orbit equivalent action $\Gamma \times \Z \actson (X',\mu')$ where the $\Gamma$-action preserves the infinite measure $\mu'$ and the $\Z$-action scales $\mu'$ by powers of $\lambda$.

{\bf Acknowledgment.} The main part of this work was done at the \emph{Institut Henri Poincar\'{e}} in Paris and I would like to thank the institute for their hospitality.

\section{Preliminaries}

In the proof of our main technical result (Theorem \ref{thm.main-tech}), we make use of the following criterion for relative amenability due to \cite{OP07} (see also \cite[Section 2.5]{PV11}). We copy the formulation of \cite[Lemma 2.3]{Io12}.

\begin{lemma}[{\cite[Corollary 2.3]{OP07}}] \label{lem.rel-amen}
Let $(M,\tau)$ be a tracial von Neumann algebra and $p \in M$ a nonzero projection. Let $A \subset pMp$ and $B \subset M$ be von Neumann subalgebras. Let $\cL$ be any $B$-$M$-bimodule.

Assume that there exists a net of vectors $(\xi_i)_{i \in I}$ in $p L^2(M) \ot_B \cL$ with the following properties.
\begin{itemize}
\item For every $x \in pMp$, we have $\limsup_{i \in I} \|x \xi_i\|_2 \leq \|x\|_2$.
\item We have $\limsup_{i \in I} \|\xi_i\|_2 > 0$.
\item For every $a \in \cU(A)$, we have that $\lim_{i \in I} \|a \xi_i - \xi_i a \|_2 = 0$.
\end{itemize}
Then there exists a nonzero projection $q$ in the center of $A' \cap pMp$ such that $Aq$ is amenable relative to $B$ inside $M$.
\end{lemma}

\section{Key technical theorem}\label{sec.key}

Throughout this section, we fix an amalgamated free product $M = M_1 *_B M_2$ of tracial von Neumann algebras $(M_i,\tau)$ with a common von Neumann subalgebra $B \subset M_i$ w.r.t.\ the unique trace preserving conditional expectations.

\subsection{The malleable deformation of an amalgamated free product}\label{sec.malleable}

We recall from \cite[Section 2.2]{IPP05} the construction of Popa's malleable deformation of $M$. We denote $G = \F_2$, with free generators $a, b \in G$. Write $G_1 = a^\Z$ and $G_2 = b^\Z$. We define $\Mtil = M *_B (B \ovt L(G))$. Writing $\Mtil_i = M_i *_B (B \ovt L(G_i))$, we can also view $\Mtil = \Mtil_1 *_B \Mtil_2$. Choose self-adjoint elements $h_j \in L(G_j)$ with spectrum $[-\pi,\pi]$ such that $u_a = \exp(i h_1)$ and $u_b = \exp(i h_2)$. Define the $1$-parameter groups of unitaries $(u_{j,t})_{t \in \R}$ in $L(G_j)$ given by $u_{j,t} = \exp(i t h_j)$. We finally define the $1$-parameter group of automorphisms $(\theta_t)_{t \in \R}$ of $\Mtil$ given by
$$\theta_t(x) = u_{j,t} x u_{j,t}^* \quad\text{for all}\;\; x \in \Mtil_j \; .$$
Note that $\theta_t$ is well defined because $u_{j,t} b u_{j,t}^* = b$ for all $b \in B$ and $j \in \{1,2\}$.

We define $\cS$ as the set of all finite alternating sequences of $1$'s and $2$'s, including the empty sequence $\emptyset$. So the elements of $\cS$ are the finite sequences $(1,2,1,2,\cdots)$ and $(2,1,2,1,\cdots)$. The length of an alternating sequence $\cI \in \cS$ is denoted by $|\cI|$. For every $(i_1,\ldots,i_n) \in \cS$, we define $\cH_{(i_1,\ldots,i_n)} \subset L^2(M)$ as the closed linear span of $(M_{i_1} \ominus B) \cdots (M_{i_n} \ominus B)$. By convention, we put $\cH_\emptyset = L^2(B)$. So we have the orthogonal decomposition
$$L^2(M) = \bigoplus_{\cI \in \cS} \cH_\cI \; .$$
We denote by $P_\cI$ the orthogonal projection of $L^2(M)$ onto $\cH_\cI$.

Denote $\rho_t = |\sin(\pi t) / \pi t|^2$. A direct computation gives us that for all $x \in L^2(M)$ and all $t \in \R$,
\begin{align}
\|E_M(\theta_t(x))\|_2^2 &= \sum_{\cI \in \cS} \rho_t^{2|\cI|} \|P_\cI(x)\|_2^2 \; , \label{eq.EMtheta}\\
\|x - \theta_t(x)\|_2^2 &= \sum_{\cI \in \cS} 2(1-\rho_t^{|\cI|}) \|P_\cI(x)\|_2^2 \; , \notag\\
\|\theta_t(x) - E_M(\theta_t(x))\|_2^2 &= \sum_{\cI \in \cS} (1- \rho_t^{2|\cI|}) \|P_\cI(x)\|_2^2 \; \notag.
\end{align}
The last two equalities imply the following transversality property in the sense of \cite[Lemma 2.1]{Po06}.
\begin{equation}\label{eq.transverse}
\|x - \theta_t(x)\|_2 \leq \sqrt{2} \|\theta_t(x) - E_M(\theta_t(x))\|_2 \quad\text{for all}\;\; x \in L^2(M), t \in \R \; .
\end{equation}

The following is the main technical result of \cite{IPP05}. For a proof of the version as we state it here, we refer to \cite[Section 5]{Ho07} and \cite[Theorem 5.4]{PV09}.

\begin{theorem}[{\cite[Theorem 3.1]{IPP05}}] \label{thm.IPP}
Let $p \in M$ be a nonzero projection and $A \subset pMp$ a von Neumann subalgebra. Assume that there exists an $\eps > 0$ and a $t > 0$ such that $\|E_M(\theta_t(a))\|_2 \geq \eps$ for all $a \in \cU(A)$. Then at least one of the following statements hold.
\begin{itemize}
\item $A \prec_M B$.
\item There exists an $i \in \{1,2\}$ such that $\cN_{pMp}(A)\dpr \prec_M M_i$.
\end{itemize}
\end{theorem}

\subsection{\boldmath The algebra $\Mtil$ as a crossed product with $\F_2$ and random walks on $\F_2$} \label{sec.ioana}

We recall here the fundamental idea of \cite{Io12} to consider $\Mtil$ as a crossed product with the free group $\F_2$ and to exploit the spectral gap of random walks on the nonamenable group $G = \F_2$. As in \cite[Remark 4.5]{Io06} and \cite[Section 3]{Io12}, we decompose $M = N \rtimes G$, where $N$ is defined as the von Neumann subalgebra of $\Mtil$ generated by $\{u_g M u_g^* \mid g \in G\}$ and normalized by the unitaries $(u_g)_{g \in G}$. Note that $N$ is the infinite amalgamated free product of the subalgebras $u_g M u_g^*$, $g \in G$, over the common subalgebra $B$. From this point of view, the action of $G$ on $N$ is the free Bernoulli action.

For every $i \in \{1,2\}$ and $t \in (0,1)$, we define the maps $\be_{i,t} : G_i \recht \R$ given by
$$\be_{i,t}(g) = \tau(u_{i,t} u_g^*) \quad\text{for all}\;\; g \in G_i \; .$$
We then denote by $\gamma_{i,t}$ and $\mu_{i,t}$ the probability measures on $G$ given by
\begin{align*}
\gamma_{i,t}(g) &= \begin{cases} |\be_{i,t}(g)|^2 &\quad\text{if}\;\; g \in G_i \; , \\ 0 &\quad\text{if}\;\; g \not\in G_i \; ,\end{cases} \\
\mu_{i,t} &= \gamma_{i,t} * \gamma_{i,t} \; ,
\end{align*}
where we used the usual convolution product between probability measures on $G$~:
$$(\gamma * \gamma')(g) = \sum_{h,k \in G, hk = g} \gamma(h) \, \gamma'(k) \; .$$

For $\cI \in \cS$, we finally denote by $\mu_{\cI,t}$ the probability measure on $G$ given by
$$\mu_{\emptyset,t}(g) = \delta_{g,e} \quad\text{and}\quad \mu_{(i_1,\ldots,i_n),t} = \mu_{i_1,t} * \mu_{i_2,t} * \cdots * \mu_{i_n,t} \; .$$
The probability measures $\mu_{\cI,t}$ give rise to the Markov operators $T_{\cI,t}$ on $\ell^2(G)$ given by
$$T_{\cI,t} = \sum_{g \in G} \mu_{\cI,t}(g) \lambda_g \; .$$
The support of the probability measures $\gamma_{i,t}$ and $\mu_{i,t}$ equals $G_i$. So the support $S$ of the probability measure $\mu_{(1,2),t}$ equals $G_1 G_2$. Since $S S^{-1}$ generates the group $\F_2$ and since $\F_2$ is nonamenable, it follows from Kesten's criterion (see e.g.\ \cite[Corollary 18.5]{Pi84}) that $\|T_{(1,2),t}\| < 1$ for all $t \in (0,1)$.
Writing $c_t = \|T_{(1,2),t}\|^{1/2}$, we have found numbers $0 < c_t < 1$ such that
$$\| T_{\cI,t} \| \leq c_t^{|\cI|-1} \quad\text{for all}\;\; \cI \in \cS \;\;\text{and all}\;\; 0 < t < 1 \; .$$

For every $x \in \Mtil$ and $h \in G$, we define $(x)_h = E_N(x u_h^*)$. So with $\|\,\cdot\,\|_2$-convergence, we have $x = \sum_{h \in G} (x)_h u_h$. We recall the following result of \cite{Io12}.

\begin{lemma}[{\cite[Formula (3.5)]{Io12}}] \label{lem.compute}
For all $t \in (0,1)$, $h \in G$ and $x,y \in L^2(M)$, we have that
$$\langle (\theta_t(x))_h , (\theta_t(y))_h \rangle = \sum_{\cI \in \cS} \langle P_\cI(x) , y \rangle \; \mu_{\cI,t}(h) \; .$$
\end{lemma}

Also recall from \cite{Io12} that Lemma \ref{lem.compute} yields the following result.

\begin{theorem}[{\cite[Theorem 3.2]{Io12}}] \label{thm.embed}
Let $p \in M$ be a nonzero projection and $A \subset pMp$ a von Neumann subalgebra. Assume that for some $t \in (0,1)$, we have that $\theta_t(A) \prec_{\Mtil} N$. Then at least one of the following statements holds.
\begin{itemize}
\item $A \prec_M B$.
\item There exists an $i \in \{1,2\}$ such that $\cN_{pMp}(A)\dpr \prec_M M_i$.
\end{itemize}
\end{theorem}
\begin{proof}
We prove the theorem by contraposition. So assume that the conclusion fails. By Theorem \ref{thm.IPP}, we find a net of unitaries $(a_i)_{i \in I}$ in $\cU(A)$ such that for all $s \in (0,1)$, we have that $\lim_{i \in I} \|E_M(\theta_s(a_i))\|_2 = 0$. We prove that for all $t \in (0,1)$, we have that $\theta_t(A) \not\prec_{\Mtil} N$. So fix $t \in (0,1)$. It suffices to prove that for all $h \in G$, we have $\lim_{i \in I} \|(\theta_t(a_i))_h\|_2 = 0$.

Fix $h \in G$ and fix $\eps > 0$. Take a large enough integer $n_0$ such that $c_t^{n_0-1} < \eps$. So, for all $\cI \in \cS$ with $|\cI| \geq n_0$, we have that $\|T_{\cI,t}\| < \eps$ and, in particular,
$$\mu_{\cI,t}(h) = \langle T_{\cI,t} \delta_e , \delta_h \rangle < \eps \; .$$
Denote by
$$P_0 = \sum_{\cI \in \cS, |\cI| < n_0} P_\cI$$
the projection onto the closed linear span of ``all words of length $< n_0$''. Using Lemma \ref{lem.compute}, we get for all $i \in I$ that
$$\|(\theta_t(a_i))_h\|_2^2 \leq \|P_0(a_i)\|_2^2 + \eps \; .$$
By \eqref{eq.EMtheta}, we can take $s > 0$ small enough such that $\|P_0(a_i)\|_2 \leq 2 \|E_M(\theta_s(a_i))\|_2$ for all $i \in I$.
Since $\lim_{i \in I} \|E_M(\theta_s(a_i))\|_2 = 0$, it follows that
$$\limsup_{i \in I} \|(\theta_t(a_i))_h\|_2^2 \leq \eps \; .$$
Since $\eps > 0$ is arbitrary, it indeed follows that for all $h \in G$, we have $\lim_{i \in I} \|(\theta_t(a_i))_h\|_2 = 0$.
\end{proof}

\subsection{Relative amenability and the malleable deformation}

The following is our main technical result. The same statement was proven in \cite[Theorem 5.1]{Io12} under the additional assumption that $A' \cap (pMp)^\omega = \C 1$ for some free ultrafilter $\omega$, i.e.\ under the assumption that there are no nontrivial bounded sequences in $pMp$ that asymptotically commute with $A$.

\begin{theorem}\label{thm.main-tech}
Let $p \in M$ be a nonzero projection and $A \subset pMp$ a von Neumann subalgebra. Assume that for all $t \in (0,1)$, we have that $\theta_t(A)$ is amenable relative to $N$ inside $\Mtil$. Then at least one of the following statements holds.
\begin{itemize}
\item There exists $i \in \{1,2\}$ such that $A \prec_M M_i$.
\item We have that $A$ is amenable relative to $B$ inside $M$.
\end{itemize}
\end{theorem}

\begin{proof}
Assume that $A \not\prec_M M_1$ and $A \not\prec_M M_2$. Denote by $z$ the maximal projection in the center of $A' \cap pMp$ such that $Az$ is amenable relative to $B$ inside $M$. If $z = p$, then the theorem is proven. If $z < p$, we replace $p$ by $p-z$ and we replace $A$ by $A(p-z)$. So, $A \not\prec_M M_i$ for all $i \in \{1,2\}$ and for all nonzero projections $q \in \cZ(A' \cap pMp)$, we have that $Aq$ is not amenable relative to $B$. We refer to this last property by saying that ``no corner of $A$ is amenable relative to $B$ inside $M$.'' We will derive a contradiction.

Exactly as in the proof of \cite[Theorem 5.1]{Io12}, we define the index set $I$ consisting of all quadruplets $i = (X,Y,\delta,t)$ where $X \subset \Mtil$ and $Y \subset \cU(A)$ are finite subsets, $\delta \in (0,1)$ and $t \in (0,1)$. We turn $I$ into a directed set by putting $(X,Y,\delta,t) \leq (X',Y',\delta',t')$ if and only if $X \subset X'$, $Y \subset Y'$, $\delta' \leq \delta$ and $t' \leq t$. Since $\theta_t(A)$ is amenable relative to $N$ inside $\Mtil$ for all $t \in (0,1)$, we can choose for every $i = (X,Y,\delta,t)$ in $I$, a vector $\xi_i \in L^2(\langle \Mtil,e_N\rangle)$ such that $\|\xi_i\|_2 \leq 1$ and
\begin{align*}
|\langle x \xi_i , \xi_i \rangle - \tau(x) | \leq \delta &\quad\text{whenever}\;\; x \in X \;\;\text{or}\;\; x = (\theta_t(y) - y)^* (\theta_t(y) - y) \;\;\text{with}\;\; y \in Y \; ,\\
\|\theta_t(y) \xi_i - \xi_i \theta_t(y)\|_2 \leq \delta &\quad\text{whenever}\;\; y \in Y \; .
\end{align*}

It follows that for all $x \in \Mtil$, we have that $\lim_{i \in I} \langle x \xi_i , \xi_i \rangle = \tau(x)$. Since for all $y \in \cU(A)$, we have that $\lim_{t \recht 0} \|\theta_t(y) - y\|_2 = 0$, it follows that for all $y \in \cU(A)$, we have that $\lim_{i \in I} \|y \xi_i - \xi_i y\|_2 = 0$.

Denote by $\cK$ the closed linear span of $\{x \, u_g e_N u_g^* \mid x \in M, g \in G\}$ inside $L^2(\langle \Mtil, e_N \rangle)$. Denote by $e$ the orthogonal projection onto $\cK$. The net of vectors $\xi'_i = p (1-e)(\xi_i)$ satisfies $\limsup_{i \in I} \|x \xi'_i\|_2 \leq \|x\|_2$ for all $x \in pMp$ and $\lim_{i \in I} \|a \xi'_i - \xi'_i a\|_2 = 0$ for all $a \in A$.
By \cite[Lemma 4.2]{Io12}, the $M$-$M$-bimodule $L^2(\langle \Mtil,e_N \rangle) \ominus \cK$ is isomorphic with $L^2(M) \ot_B \cL$ for some $B$-$M$-bimodule $\cL$. Since no corner of $A$ is amenable relative to $B$ inside $M$, it follows from Lemma \ref{lem.rel-amen} that $\lim_{i \in I} \|\xi'_i\|_2 = 0$. So,
$$\lim_{i \in I} \|p \xi_i - e(p \xi_i)\|_2 = 0 \; .$$
Define the isometry
$$U : L^2(M) \ot \ell^2(G) \recht L^2(\langle \Mtil, e_N \rangle) : U (x \ot \delta_g) = x \, u_g e_N u_g^* \; .$$
Note that $UU^* = e$ and
$$U((x \ot 1)\eta(y \ot 1)) = x U(\eta) y \quad\text{for all}\;\; x,y \in M, \eta \in L^2(M) \ot \ell^2(G) \; .$$
We define the net of vectors $(\zeta_i)_{i \in I}$ in $p L^2(M) \ot \ell^2(G)$ given by $\zeta_i = U^*(p \xi_i)$. Note that $\|\zeta_i\|_2 \leq 1$. The properties of the net $(\xi_i)_{i \in I}$ imply that
\begin{alignat*}{2}
\lim_{i \in I} \| p \xi_i - U(\zeta_i) \|_2 &= 0 \; , &&\\
\lim_{i \in I} \langle (x \ot 1) \zeta_i , \zeta_i \rangle &= \tau(pxp) & &\quad\text{for all}\;\; x \in  M  \; ,\\
\lim_{i \in I} \| (a \ot 1) \zeta_i - \zeta_i (a \ot 1) \|_2 &= 0 & &\quad\text{for all}\;\; a \in \cU(A) \; .
\end{alignat*}

We view $p L^2(M) \ot \ell^2(G)$ as a closed subspace of $L^2(\Mtil) \ot \ell^2(G)$. As such, the following claim makes sense.

{\bf Claim.} {\it For every $\eps > 0$, there exists an $s_0 \in (0,1)$ and an $i_0 \in I$ such that}
$$\|\zeta_i - (\theta_s \ot \id)(\zeta_i)\|_2 < \eps \quad\text{\it for all}\;\; s \in [0,s_0] \;\;\text{\it and all}\;\; i \geq i_0 \; .$$

{\bf Proof of the claim.} Assume the contrary. Using \eqref{eq.transverse}, we then find an $\eps > 0$ such that for every $s \in (0,1)$, we have
$$\limsup_{i \in I} \|(\theta_s \ot \id)(\zeta_i) - (E_M \circ \theta_s \ot \id)(\zeta_i)\|_2 \geq \eps \; .$$
Since for every $a \in \cU(A)$, we have $\lim_{s \recht 0} \|\theta_s(a) - a \|_2 = 0$, we can choose a subnet $(\mu_k)$ of the net of vectors
$$\Bigl( (p \ot 1) ( (\id - E_M) \circ \theta_s \ot \id)(\zeta_i)\Bigr)_{(i,s) \in I \times (0,1)}$$
with the properties that $\limsup_k \|(x \ot 1) \mu_k\|_2 \leq \|x\|_2$ for all $x \in pMp$, $\liminf_k \|\mu_k\|_2 \geq \eps$ and $\lim_k \|(a \ot 1)\mu_k - \mu_k (a \ot 1)\|_2 = 0$ for all $a \in \cU(A)$. The $M$-$M$-bimodule $L^2(\Mtil \ominus M) \ot \ell^2(G)$ is isomorphic with $L^2(M) \ot_B \cL$ for some $B$-$M$-bimodule $\cL$. By Lemma \ref{lem.rel-amen}, we reach a contradiction with the assumption that no corner of $A$ is amenable relative to $B$ inside $M$. This proves the claim.

Put $\eps = \tau(p) / 14$. Fix $i_0 \in I$ and $s_0 \in (0,1)$ such that for all $i \geq i_0$ and all $s \in [0,s_0]$, we have that
$$\|p \xi_i - U(\zeta_i)\|_2 < \eps \quad\text{and}\quad \|\zeta_i - (\theta_s \ot \id)(\zeta_i)\|_2 < \eps \; .$$
Write $i_0 = (X_0,Y_0,\delta_0,t_0)$. Enlarging $i_0$ if necessary, we may assume that $p \in X_0$, that $p \in Y_0$ (note that $p$ is the unit element of $\cU(A)$), that $\delta_0 < \eps^2/2$, that $t_0 \leq s_0$ and that $\|\theta_{t_0}(p) - p\|_2 < \eps/2$.

Denote by $J$ the index set consisting of all triplets $j = (X,Y,\delta)$, where $X \subset p M p$ and $Y \subset \cU(A)$ are finite subsets and $\delta \in (0,\delta_0)$. We turn $J$ into a directed set in a similar way as $I$ above. For every $j = (X,Y,\delta)$, we put
$$\eta_j = \zeta_{(X_0 \cup X, Y_0 \cup Y, \delta, t_0)} \; .$$
Note that we use here the fixed index $t_0$. In particular, $(\eta_j)_{j \in J}$ is not a subnet of $(\zeta_i)_{i \in I}$. Also note that $\|\eta_j\|_2 \leq 1$.
We claim that the net $(\eta_j)_{j \in J}$ of vectors in $p L^2(M) \ot \ell^2(G)$ has the following properties.
\begin{alignat}{2}
& \limsup_{j \in J} \| (x \ot 1)\eta_j \|_2 \leq \|x\|_2 &\quad& \text{for all}\;\; x \in M \; , \label{eq.1}\\
& \liminf_{j \in J} |\langle \theta_{t_0}(a) \, U(\eta_j) , U(\eta_j) \, \theta_{t_0}(a) \rangle | \geq \tau(p) - 6\eps &&\text{for all}\;\; a \in \cU(A) \; ,\label{eq.2}\\
& \|\eta_j - (\theta_s \ot \id)(\eta_j)\|_2 \leq \eps &&\text{for all}\;\; s \in [0,t_0] \;\text{and}\; j \in J \; .\label{eq.3}
\end{alignat}
To prove \eqref{eq.1}, fix $x \in M$ and fix $j = (X,Y,\delta)$ with $p x^* x p \in X$. It suffices to prove that
\begin{equation}\label{eq.hulp}
\|(x \ot 1) \eta_j\|_2^2 \leq \|x\|_2^2 + \delta \; .
\end{equation}
Put $i = (X_0 \cup X, Y_0 \cup Y, \delta, t_0)$. We get that
$$
\|(x \ot 1)\eta_j\|_2 = \|(x \ot 1) \zeta_i\|_2 = \|x \, U(\zeta_i)\|_2 = \|x e(p \xi_i)\|_2 = \|e (x p \xi_i)\|_2 \leq \|x p \xi_i\|_2 \; .
$$
But also
$$\|x p \xi_i\|_2^2 = \langle p x^* x p \xi_i , \xi_i \rangle \leq \tau(p x^* x p) + \delta \leq \|x\|_2^2 + \delta$$
because $p x^* x p \in X \subset X_0 \cup X$. So \eqref{eq.hulp} follows and \eqref{eq.1} is proven.

To prove \eqref{eq.2}, fix $a \in \cU(A)$ and fix $j = (X,Y,\delta)$ with $a \in Y$. It suffices to prove that
\begin{equation}\label{eq.hulp2}
|\langle \theta_{t_0}(a) \, U(\eta_j) , U(\eta_j) \, \theta_{t_0}(a) \rangle | \geq \tau(p) - 6\eps - 2\delta \; .
\end{equation}
Put $i = (X_0 \cup X, Y_0 \cup Y, \delta, t_0)$. Since $p \in Y_0 \subset Y_0 \cup Y$, we have
$$\|\theta_{t_0}(p) \xi_i - p \xi_i\|_2^2 \leq \|\theta_{t_0}(p) - p \|_2^2 + \delta \leq \frac{\eps^2}{2} + \delta_0 \leq \eps^2 \; .$$
So $\|\theta_{t_0}(p) \xi_i - p \xi_i\|_2 \leq \eps$. Since $\|p \xi_i - U(\zeta_i)\|_2 \leq \eps$, we get that
$$\|\theta_{t_0}(p) \xi_i - U(\eta_j)\|_2 \leq 2\eps \; .$$
Since $p \in Y_0 \subset Y_0 \cup Y$, we also have that $\|\theta_{t_0}(p) \xi_i - \xi_i \theta_{t_0}(p)\|_2 \leq \delta \leq \delta_0 \leq \eps$. In combination with the previous inequality, this gives
$$\|\xi_i \theta_{t_0}(p) - U(\eta_j)\|_2 \leq 3 \eps \; .$$
In the following computation, we write $y \approx_\eps z$ when $y,z \in \C$ with $|y - z| \leq \eps$. We also use throughout that $\|\zeta_i\|_2 \leq 1 $ and $\|\eta_j\|_2 \leq 1$ for all $i \in I$ and $j \in J$. So,
\begin{alignat*}{2}
\langle \theta_{t_0}(a) \,&  U(\eta_j) , U(\eta_j) \, \theta_{t_0}(a) \rangle &\qquad& \\
& \approx_{2\eps} \langle \theta_{t_0}(a) \xi_i , U(\eta_j) \, \theta_{t_0}(a) \rangle && \text{because}\;\; \|U(\eta_j) - \theta_{t_0}(p) \xi_i\|_2 \leq 2\eps \; ,\\
& \approx_{3\eps} \langle \theta_{t_0}(a) \xi_i , \xi_i \theta_{t_0}(a) \rangle && \text{because}\;\; \|U(\eta_j) - \xi_i \theta_{t_0}(p)\|_2 \leq 3\eps \; ,\\
& \approx_\delta \langle \theta_{t_0}(a) \xi_i , \theta_{t_0}(a) \xi_i \rangle && \text{because}\;\; \|\xi_i \theta_{t_0}(a) - \theta_{t_0}(a) \xi_i\|_2 \leq \delta \;\;\text{since}\;\; a \in Y \; ,\\
& = \langle \theta_{t_0}(p) \xi_i , \xi_i \rangle && \\
& \approx_\eps \langle p \xi_i , \xi_i \rangle && \text{because}\;\; \|\theta_{t_0}(p) \xi_i - p \xi_i\|_2 \leq \eps \; ,\\
& \approx_\delta \tau(p) && \text{because}\;\; p \in X_0 \subset X_0 \cup X \; .
\end{alignat*}
From this computation, \eqref{eq.hulp2} follows immediately. So also \eqref{eq.2} is proven.

Finally \eqref{eq.3} follows because $\|\zeta_i - (\theta_s \ot \id)(\zeta_i)\|_2 \leq \eps$ for all $s \in [0,t_0]$ and all $i \geq i_0$.

Denote $\eta_j = \sum_{g \in G} \eta_{j,g} \ot \delta_g$, where $\eta_{j,g} \in L^2(M)$ and where $(\delta_g)_{g \in G}$ is the canonical orthonormal basis of $\ell^2(G)$. Recall that for every $x \in L^2(\Mtil)$ and $h \in G$, we denote $(x)_h = E_N(x u_h^*)$.

For every $a \in \cU(A)$, we have that
$$\sum_{g,h \in G} \|(\theta_{t_0}(a \eta_{j,g}))_h\|_2^2 = \sum_{g \in G} \|\theta_{t_0}(a \eta_{j,g})\|_2^2 = \sum_{g \in G} \|a \eta_{j,g}\|_2^2  = \|(a \ot 1) \eta_j\|_2^2 \leq 1 \; .$$
Because the subspaces $(L^2(N) u_{hg} e_N u_g^*)_{h,g \in G}$ of $L^2(\langle \Mtil, e_N \rangle)$ are orthogonal, the formula
$$\xi(a,j) = \sum_{g,h \in G} (\theta_{t_0}(a \eta_{j,g}))_h \, u_{hg} e_N u_g^*$$
provides a well defined vector in $L^2(\langle \Mtil , e_N \rangle)$ with $\|\xi(a,j)\|_2 \leq 1$. We claim that for every $a \in \cU(A)$ and all $j \in J$, we have
\begin{equation}\label{eq.estimate}
\| \theta_{t_0}(a) \, U(\eta_j) - \xi(a,j) \|_2 \leq \eps \; .
\end{equation}
This follows because
\begin{align*}
\| \theta_{t_0}(a) \, U(\eta_j) - \xi(a,j) \|_2^2 &= \sum_{g,h \in G} \|(\theta_{t_0}(a) \eta_{j,g})_h - (\theta_{t_0}(a \eta_{j,g}))_h\|_2^2 \\
&= \sum_{g \in G} \| \theta_{t_0}(a) \eta_{j,g} - \theta_{t_0}(a \eta_{j,g})\|_2^2 \\
&= \sum_{g \in G} \|\eta_{j,g} - \theta_{t_0}(\eta_{j,g})\|_2^2 \\
&= \|\eta_j - (\theta_{t_0} \ot \id)(\eta_j)\|_2^2 \leq \eps^2 \; .
\end{align*}
So \eqref{eq.estimate} is proven.

We similarly define the vectors $\xi'(a,j) \in L^2(\langle \Mtil,e_N \rangle)$ by the formula
$$\xi'(a,j) = \sum_{g,h \in G} (\theta_{t_0}(\eta_{j,g} a))_h \; u_{g} e_N u_g^*u_h = \sum_{g,h \in G} (\theta_{t_0}(\eta_{j,hg} a))_h \; u_{hg} e_N u_g^*$$
and get that $\|\xi'(a,j)\|_2 \leq 1$ and
\begin{equation}\label{eq.est2}
\| U(\eta_j) \, \theta_{t_0}(a) - \xi'(a,j) \|_2 \leq \eps
\end{equation}
for all $a \in \cU(A)$ and all $j \in J$.

Combining \eqref{eq.hulp2}, \eqref{eq.estimate} and \eqref{eq.est2}, we find that for all $a \in \cU(A)$,
$$\limsup_{j \in J} |\langle \xi(a,j) , \xi'(a,j)\rangle| \geq \tau(p) - 8\eps \; .$$
We now apply Lemma \ref{lem.compute} and the notation introduced before the formulation of that lemma. For every $a \in \cU(A)$ and $j \in J$, we have that
\begin{align*}
\langle \xi(a,j) , \xi'(a,j)\rangle &= \sum_{g,h \in G} \langle (\theta_{t_0}(a \eta_{j,g}))_h , (\theta_{t_0}(\eta_{j,hg} a))_h \rangle \\
&= \sum_{g,h \in G} \sum_{\cI \in \cS} \langle P_{\cI}(a \eta_{j,g}), \eta_{j,hg} a \rangle \; \mu_{\cI,t_0}(h) \\
&= \langle Q_{t_0}((a \ot 1)\eta_j), \eta_j (a \ot 1)\rangle \; ,
\end{align*}
where $Q_{t_0} \in B(L^2(M) \ot \ell^2(G))$ is defined by
$$Q_{t_0} = \sum_{\cI \in \cS} P_{\cI} \ot T_{\cI,t_0} \; .$$
So we get that for all $a \in \cU(A)$,
\begin{equation}\label{eq.tocontra}
\limsup_{j \in J} |\langle Q_{t_0}((a \ot 1)\eta_j) , \eta_j (a \ot 1) \rangle| \geq \tau(p) - 8 \eps \; .
\end{equation}

Fix a large enough integer $n_0$ such that $c_{t_0}^{n_0-1} \leq \eps$. So, $\|T_{\cI,t_0}\| \leq \eps$ whenever $|\cI| \geq n_0$. Denote by
$$P_0 = \sum_{\cI \in \cS, |\cI| < n_0} P_\cI$$
the projection onto the closed linear span of ``all words of length $< n_0$''.

We claim that there exists a unitary $a \in \cU(A)$ such that
\begin{equation}\label{eq.claima}
\limsup_{j \in J} \|(P_0 \ot 1)((a \ot 1)\eta_j)\|_2 \leq 4\eps \; .
\end{equation}
To prove this claim, we first use \eqref{eq.EMtheta} to fix $0 < s \leq t_0$ close enough to zero such that
$$\|(P_0 \ot 1)(\eta)\|_2 \leq 2 \|(E_M \ot \id)((\theta_s \ot \id)(\eta))\|_2 \quad\text{for all}\;\; \eta \in L^2(M) \ot \ell^2(G) \; .$$
Since $A \not\prec M_1$ and $A \not\prec M_2$, it follows from Theorem \ref{thm.IPP} that we can choose $a \in \cU(A)$ such that $\|E_M(\theta_s(a))\|_2 \leq \eps$. We prove that this unitary $a \in \cU(A)$ satisfies \eqref{eq.claima}.

From \eqref{eq.3}, we know that $\|\eta_j - (\theta_s \ot \id)(\eta_j)\|_2 \leq \eps$ for all $j \in J$. It follows that
$$\|(\theta_s \ot \id)((a \ot 1)\eta_j) - (\theta_s(a) \ot 1) \eta_j\|_2 \leq \eps \quad\text{for all}\;\; j \in J \; .$$
So for all $j \in J$, we get that
\begin{align*}
\|(P_0 \ot 1)((a \ot 1)\eta_j)\|_2 & \leq 2 \|(E_M \ot \id)((\theta_s \ot \id)((a \ot 1)\eta_j))\|_2 \\
& \leq 2 \|(E_M \ot \id)((\theta_s(a) \ot 1) \eta_j)\|_2 + 2 \eps \\
& = 2 \|(E_M(\theta_s(a)) \ot 1) \eta_j \|_2 + 2 \eps \; .
\end{align*}
Using \eqref{eq.1}, we get that
$$\limsup_{j \in J} \|(P_0 \ot 1)((a \ot 1)\eta_j)\|_2 \leq 2 \|E_M(\theta_s(a))\|_2 + 2 \eps \leq 4 \eps \; .$$
So the claim in \eqref{eq.claima} is proven and we fix the unitary $a \in \cU(A)$ satisfying \eqref{eq.claima}.

We now deduce that
\begin{equation}\label{eq.almost-contra}
\limsup_{j \in J} \|Q_{t_0}((a \ot 1)\eta_j)\|_2 \leq 5 \eps \; .
\end{equation}
Indeed, since $\|T_{\cI,t_0}\| \leq 1$ for all $\cI \in \cS$ and $\|T_{\cI,t_0}\| \leq \eps$ for all $\cI \in \cS$ with $|\cI| \geq n_0$, we get that
\begin{align*}
\|Q_{t_0}((a \ot 1)\eta_j)\|_2^2 &= \sum_{\cI \in \cS} \|(P_\cI \ot T_{\cI,t_0})((a \ot 1)\eta_j)\|_2^2 \\
&\leq \sum_{\cI \in \cS, |\cI| < n_0} \|(P_\cI \ot 1)((a \ot 1)\eta_j)\|_2^2 + \eps^2 \sum_{\cI \in \cS, |\cI| \geq n_0} \|(P_\cI \ot 1)((a \ot 1)\eta_j)\|_2^2 \\
&\leq \|(P_0 \ot 1)((a \ot 1)\eta_j)\|_2^2 + \eps^2 \|(a \ot 1)\eta_j\|_2^2 \; .
\end{align*}
Taking the $\limsup$ over $j \in J$ and using \eqref{eq.claima} and \eqref{eq.1}, we arrive at
$$\limsup_{j \in J} \|Q_{t_0}((a \ot 1)\eta_j)\|_2^2 \leq 17 \eps^2$$
and \eqref{eq.almost-contra} follows. But \eqref{eq.almost-contra} implies that
$$\limsup_{j \in J} |\langle Q_{t_0}((a \ot 1)\eta_j) , \eta_j (a \ot 1) \rangle| \leq 5\eps \; .$$
Since $\eps = \tau(p)/14$, we have $5\eps < \tau(p) - 8 \eps$ and so we obtained a contradiction with \eqref{eq.tocontra}.
\end{proof}

\section{Proof of Theorem \ref{thm.AFP} and a version for HNN extensions}

\begin{proof}[Proof of Theorem \ref{thm.AFP}]
We use the malleable deformation $\theta_t$ of $M \subset \Mtil$ as explained in Section \ref{sec.malleable}. Write $G = \F_2$ and $\Mtil = N \rtimes G$ as in Section \ref{sec.ioana}. By assumption, $A$ is amenable relative to one of the $M_i$ inside $M$. A fortiori, $A$ is amenable relative to $M_i$ inside $\Mtil$. Fix $t \in (0,1)$. Applying $\theta_t$, we get that $\theta_t(A)$ is amenable relative to $\theta_t(M_i)$ inside $\Mtil$. Since $\theta_t(M_i)$ is unitarily conjugate to $M_i$ and $M_i \subset N$, it follows that $\theta_t(A)$ is amenable relative to $N$ inside $\Mtil$.

Put $P := \cN_{pMp}(A)\dpr$. We apply \cite[Theorem 1.6 and Remark 6.3]{PV11} to the crossed product $\Mtil = N \rtimes G$ and the subalgebra $\theta_t(A)$ of this crossed product. We conclude that at least one of the following statements holds: $\theta_t(A) \prec_{\Mtil} N$ or $\theta_t(P)$ is amenable relative to $N$ inside $\Mtil$. Since this holds for every $t \in (0,1)$, we get that at least one of the following is true.
\begin{itemize}
\item There exists a $t \in (0,1)$ such that $\theta_t(A) \prec_{\Mtil} N$.
\item For every $t \in (0,1)$, we have that $\theta_t(P)$ is amenable relative to $N$ inside $\Mtil$.
\end{itemize}
In the first case, Theorem \ref{thm.embed} implies that $A \prec_M B$ or $P \prec_M M_i$ for some $i \in \{1,2\}$. In the second case, Theorem \ref{thm.main-tech} implies that $P \prec_M M_i$ for some $i \in \{1,2\}$, or that $P$ is amenable relative to $B$ inside $M$.
\end{proof}

By \cite{Ue07}, HNN extensions can be viewed as corners of amalgamated free products. Since Theorem \ref{thm.AFP} has no particular assumptions on the inclusions $B \subset M_i$, we can immediately deduce the following result.

\begin{theorem}\label{thm.HNN}
Let $M = \HNN(M_0,B,\theta)$ be the HNN extension of the tracial von Neumann algebra $(M_0,\tau)$ with von Neumann subalgebra $B \subset M_0$ and trace preserving embedding $\theta : B \recht M_0$. Let $p \in M$ be a nonzero projection and $A \subset pMp$ a von Neumann subalgebra that is amenable relative to $M_0$ inside $M$. Then at least one of the following statements holds.
\begin{itemize}
\item $A \prec_M B$.
\item $\cN_{pMp}(A)\dpr \prec_M M_0$.
\item We have that $\cN_{pMp}(A)\dpr$ is amenable relative to $B$ inside $M$.
\end{itemize}
\end{theorem}

\begin{proof}
By \cite[Proposition 3.1]{Ue07}, we can view $M = \HNN(M_0,B,\theta)$ as a corner of an amalgamated free product. More precisely, we put $M_1 = \M_2(\C) \ot M_0$ and $M_2 = \M_2(\C) \ot B$. We consider $B_0 = B \oplus B$ as a subalgebra of both $M_1$ and $M_2$, where the embedding $B_0 \hookrightarrow M_2$ is diagonal and the embedding $B_0 \hookrightarrow M_1$ is given by $b \oplus d \mapsto b \oplus \theta(d)$. We denote by $e_{ij}$ the matrix units in $M_1$ and by $f_{ij}$ the matrix units in $M_2$.
The HNN extension $M$ is generated by $M_0$ and the stable unitary $u$. There is a unique surjective $*$-isomorphism
$$\Psi : \HNN(M_0,B,\theta) \recht e_{11} \bigl( M_1 *_{B_0} M_2 \bigr) e_{11} : \begin{cases} \Psi(x) = e_{11}x &\;\;\text{for all}\;\; x \in M_0 \; ,\\ \Psi(u) = e_{12} f_{21} \; .\end{cases}$$
Note that in the amalgamated free product, $e_{11} = f_{11}$ and $e_{22} = f_{22}$. Therefore $e_{12} f_{21}$ is really a unitary.

Denote $\cM := M_1 *_{B_0} M_2$. Whenever $Q \subset pMp$ is a von Neumann subalgebra, one checks that
\begin{itemize}
\item $Q \prec_M B$ iff $\Psi(Q) \prec_\cM B_0$ iff $\Psi(Q) \prec_\cM M_2$.
\item $Q \prec_M M_0$ iff $\Psi(Q) \prec_M M_1$.
\item $Q$ is amenable relative to $B$ inside $M$ iff $\Psi(Q)$ is amenable relative to $B_0$ inside $\cM$.
\end{itemize}
So Theorem \ref{thm.HNN} is a direct consequence of Theorem \ref{thm.AFP}.
\end{proof}

\section{Cartan-rigidity for amalgamated free product groups and HNN extensions}

Recall from \cite{PV11} that a countable group $\Gamma$ is called $\cC$-rigid if for every free ergodic pmp action $\Gamma \actson (X,\mu)$, we have that $L^\infty(X)$ is the unique Cartan subalgebra of $L^\infty(X) \rtimes \Gamma$ up to unitary conjugacy. Observe that $\cC$-rigidity is an immediate consequence of the following stronger property $(\ast)$~:
\begin{equation}\tag{$\ast$}
\parbox[c]{13.5cm}{For every trace preserving action $\Gamma \actson (B,\tau)$, projection $p \in M = B \rtimes \Gamma$ and amenable von Neumann subalgebra $A \subset pMp$ with $\cN_{pMp}(A)\dpr = pMp$, we have that $A \prec B$.}
\end{equation}
As was shown in the proof of \cite[Theorem 1.1]{PV12}, a direct product $\Gamma_1 \times \cdots \times \Gamma_n$ of finitely many groups $\Gamma_i$ with property $(\ast)$, is $\cC$-rigid.

Property $(\ast)$ was shown to hold, among other groups, for all weakly amenable $\Gamma$ with $\beta_1^{(2)}(\Gamma) > 0$ in \cite[Theorem 7.1]{PV11} and for all nonelementary hyperbolic $\Gamma$ in \cite[Theorem 1.4]{PV12}. In \cite[Theorem 7.1]{Io12}, property $(\ast)$ was proven for a large class of amalgamated free products and in \cite[Proof of Theorem 8.1]{DI12}, for a large class of HNN extensions. For completeness, we show how to deduce these last two results from Theorem \ref{thm.AFP}, resp.\ Theorem \ref{thm.HNN}.

\begin{theorem}\label{thm.Crigid}
The following groups satisfy property $(\ast)$ and, in particular, are $\cC$-rigid.
\begin{enumerate}
\item (\cite[Theorem 7.1]{Io12})\hspace{2ex} Amalgamated free products $\Gamma = \Gamma_1 *_\Sigma \Gamma_2$ with $[\Gamma_1 : \Sigma] \geq 3$ and $[\Gamma_2 : \Sigma] \geq 2$ that admit elements $g_1,\ldots,g_n \in \Gamma$ with $\bigl|\bigcap_{k=1}^n g_k \Sigma g_k^{-1}\bigr| < \infty$.

\item (\cite[Proof of Theorem 8.1]{DI12})\hspace{2ex} HNN extensions $\Gamma = \HNN(\Gamma_1,\Sigma,\theta)$, given by a subgroup $\Sigma < \Gamma_1$ and an injective group homomorphism $\theta : \Sigma \recht \Gamma_1$, such that $\Sigma \neq \Lambda \neq \theta(\Sigma)$ and such that there exist $g_1,\ldots,g_n \in \Gamma$ with $\bigl|\bigcap_{k=1}^n g_k \Sigma g_k^{-1}\bigr| < \infty$.
\end{enumerate}
\end{theorem}
\begin{proof}
Let $\Gamma \actson (B,\tau)$ be a trace preserving action and put $M = B \rtimes \Gamma$. Let $p \in M$ be a projection and $A \subset pMp$ an amenable von Neumann subalgebra with $\cN_{pMp}(A)\dpr = pMp$. In the first case, $M$ is the amalgamated free product of $B \rtimes \Gamma_1$ and $B \rtimes \Gamma_2$ over $B \rtimes \Sigma$. In the second case, $M$ is the HNN extension of $B \rtimes \Gamma_1$ over $B \rtimes \Sigma$. In both cases, $\Gamma_i < \Gamma$ has infinite index and $\Sigma < \Gamma$ is not co-amenable. So it follows from Theorem \ref{thm.AFP} and Theorem \ref{thm.HNN} that $A \prec B \rtimes \Sigma$.

Define the projection $z(\Sigma) \in M \cap (B \rtimes \Sigma)'$ as in \cite[Section 4]{HPV10}. Since $A \prec B \rtimes \Sigma$, we have that $z(\Sigma) \neq 0$. From \cite[Proposition 8]{HPV10}, we know that $z(\Sigma)$ belongs to the center of $M$. Take $g_1,\ldots,g_n \in \Gamma$ such that $\Sigma_0 = \bigcap_{k=1}^n g_k \Sigma g_k^{-1}$ is a finite group. We have $z(g_k \Sigma g_k^{-1}) = u_{g_k} z(\Sigma) u_{g_k}^* = z(\Sigma)$, because $z(\Sigma)$ belongs to the center of $M$. It then follows from \cite[Proposition 6]{HPV10} that
$$z(\Sigma_0) = z(g_1 \Sigma g_1^{-1}) \cdots z(g_n \Sigma g_n^{-1}) = z(\Sigma) \neq 0 \; .$$
So $A \prec B \rtimes \Sigma_0$. Since $\Sigma_0$ is finite, we conclude that $A \prec B$.
\end{proof}

\section{Proof of Corollary \ref{cor.product-equiv}}

\begin{proof}[Proof of Corollary \ref{cor.product-equiv}]
Write $M_i = L(\cR_i)$ and $B = L^\infty(X)$. Note that $L(\cR) = M_1 *_B M_2$. Corollary \ref{cor.product-equiv} is a direct consequence of Theorem \ref{thm.AFP}, provided that we prove the following two statements.
\begin{enumerate}
\item $M \not\prec_M M_i$.
\item $M$ is not amenable relative to $B$, i.e.\ $M$ is not amenable itself.
\end{enumerate}
Since $|\cR_1 \cdot x| \geq 3$ for a.e.\ $x \in X$ and using e.g.\ \cite[Lemma 2.6]{IKT08}, we can take unitaries $u,v \in \cU(M_1)$ such that $E_B(u) = E_B(v) = E_B(u^* v) = 0$.
We similarly find a unitary $w \in \cU(M_2)$ with $E_B(w) = 0$.

Proof of 1. Define the sequence of unitaries $w_n \in \cU(M)$ given by $w_n = (uw)^n$. Denote by $X_m \subset M$ the linear span of all products of at most $m$ elements from $M_1 \ominus B$ and $M_2 \ominus B$. Whenever $2n > 2m+1$ and $x,y \in X_m$, a direct computation yields that $E_{M_i}(x w_n y) = 0$. So it follows that $\lim_n \|E_{M_i}(x w_n y)\|_2 = 0$ for all $x,y \in M$ and 1 follows.

Proof of 2. Assume that $M$ is amenable and take an $M$-central state $\Omega$ on $B(L^2(M))$. Define $K_1$ as the closed linear span of $B$ and all products of the form $x_1 x_2 \cdots x_n$ with $x_1 \in M_1 \ominus B$, $x_2 \in M_2 \ominus B$, etc. Define $K_2$ as the closed linear span of all products of the form $y_1 y_2 \cdots y_n$ with $y_1 \in M_2 \ominus B$, $y_2 \in M_1 \ominus B$, etc. By construction, $L^2(M) = K_1 \oplus K_2$. Denote by $e_i$ the orthogonal projection of $L^2(M)$ onto $K_i$. It follows that $u e_2 u^*$ and $v e_2 v^*$ are orthogonal and lie under $e_1$. Hence,
$2 \Omega(e_2) = \Omega(u e_2 u^*) + \Omega(v e_2 v^*) \leq \Omega(e_1)$. On the other hand, $w e_1 w^* \leq e_2$, implying that $\Omega(e_1) = \Omega(w e_1 w^*) \leq \Omega(e_2)$. Altogether it follows that $\Omega(e_1) = \Omega(e_2) = 0$. Since $1 = e_1 + e_2$ and $\Omega(1) = 1$, we reached a contradiction.
\end{proof}

\section{Proof of Corollary \ref{cor.strong-solid}}

\begin{proof}[Proof of Corollary \ref{cor.strong-solid}]
Let $A \subset M$ be a diffuse amenable von Neumann subalgebra. Denote $P = \cN_M(A)\dpr$ and assume that $P$ is not amenable. Take a nonzero central projection $z \in \cZ(P)$ such that $Pz$ has no amenable direct summand. Since $Pz \subset \cN_{zMz}(Az)\dpr$, it follows from Theorem \ref{thm.AFP} that one of the following statements holds.
\begin{enumerate}
\item $Az \prec_M B$.
\item $Pz \prec_M M_i$ for some $i \in \{1,2\}$.
\item $Pz$ is amenable relative to $B$ inside $M$.
\end{enumerate}
It suffices to prove that each of the three statements is false.

1. Observe that the inclusion $M_2 \subset M$ is mixing. To prove this, fix a sequence $b_n$ in the unit ball of $M_2$ such that $b_n \recht 0$ weakly. We must show that $\lim_n \|E_{M_2}(x^* b_n y)\|_2 = 0$ for all $x,y \in M \ominus M_2$. It suffices to prove this when $x=x_1 x_2 \cdots x_n$ and $y=y_1 y_2 \cdots y_m$ with $n,m \geq 2$, $x_1,y_1 \in M_2$, $x_2,y_2 \in M_1 \ominus B$, $x_3,y_3 \in M_2 \ominus B$, etc. But then
$$E_{M_2}(x^* b_n y) = E_{M_2}(x_n^* \cdots x_3^* E_B(x_2^* E_B(x_1^* b_n y_1) y_2) y_3 \cdots y_n)$$
and the conclusion follows because $E_B(x_1^* b_n y_1) \recht 0$ weakly and the inclusion $B \subset M_1$ is mixing.

Assume that 1 holds. Then certainly $Az \prec_M M_2$. Since the inclusion $M_2 \subset M$ is mixing, it follows from \cite[Lemma 9.4]{Io12} that $Pz \prec_M M_2$. So statement 2 holds and we proceed to the next point.

2. Observe that 2 holds. We then find in particular a nonzero projection $p \in M_n(\C) \ot M_i$ and a normal unital $*$-homomorphism $\vphi : Pz \recht p(M_n(\C) \ot M_i)p$. Then $\vphi(Az)$ is a diffuse von Neumann subalgebra of $p(M_n(\C) \ot M_i)p$ whose normalizer contains $\vphi(Pz)$. Since $Pz$ has no amenable direct summand, $\vphi(Pz)$ is nonamenable. Hence $p(M_n(\C) \ot M_i)p$ is not strongly solid. Since $M_i$ is strongly solid, this contradicts  the stability of strong solidity under amplifications as proven in \cite[Proposition 5.2]{Ho09}.

3. Since $B$ is amenable, 3 implies that $Pz$ is amenable, contradicting our assumptions.
\end{proof}

\section{\boldmath W$^*$-superrigid actions of type III}\label{sec.Wstar-superrigid}

In the same way as \cite[Theorem A]{HV12} was deduced from the results in \cite{PV12}, we can deduce from Theorem \ref{thm.AFP} the following type III uniqueness statement for Cartan subalgebras. Our theorem is a generalization of \cite[Theorem D]{BHR12}, where the same result was proven under the assumption that $\Sigma$ is a finite group.

Rather than looking for the most general statement possible, we provide a more ad hoc formulation that suffices to prove the W$^*$-superrigidity of the type III$_1$ actions in Proposition \ref{prop.Wstar-superrigid} (see also Remark \ref{rem.more-general} below). Recall that a nonsingular action $\Lambda \actson (X,\mu)$ is said to be recurrent if there is no Borel subset $\cU \subset X$ such that $\mu(\cU) > 0$ and $\mu(g \cdot \cU \cap \cU) = 0$ for all $g \in \Lambda - \{e\}$.

\begin{theorem}\label{thm.unique-cartan-type-III}
Let $\Gamma = \Gamma_1 *_\Sigma \Gamma_2$ be an amalgamated free product group and assume that there exist $g_1,\ldots,g_n \in \Gamma$ such that $\bigcap_{k=1}^n g_k \Sigma g_k^{-1}$ is finite.

Let $\Gamma \actson (X,\mu)$ be any nonsingular free ergodic action. Assume that both $\Gamma_i$ admit a subgroup $\Lambda_i$ such that the restricted action $\Lambda_i \actson (X,\mu)$ is recurrent and $\Lambda_i \cap \Sigma$ is finite.

Then, $L^\infty(X)$ is the unique Cartan subalgebra of $L^\infty(X) \rtimes \Gamma$ up to unitary conjugacy.
\end{theorem}

For Theorem \ref{thm.unique-cartan-type-III} to hold, it is essential to impose some recurrence of $\Gamma_i \actson (X,\mu)$ relative to $\Sigma$. Indeed, otherwise the action $\Gamma \actson (X,\mu)$ could simply be the induction of an action $\Gamma_i \actson (Z,\eta)$ so that $L^\infty(X) \rtimes \Gamma \cong B(H) \ovt (L^\infty(Z) \rtimes \Gamma_i)$ and we cannot expect uniqueness of the Cartan subalgebra.

Before proving Theorem \ref{thm.unique-cartan-type-III}, we provide a semifinite variant of the machinery developed in \cite[Sections 4 and 5]{HPV10}. We start by the following elementary lemma and leave the proof to the reader.

\begin{lemma}\label{lem.dim}
Let $(N,\Tr)$ be a von Neumann algebra equipped with a normal semifinite faithful trace. Let $H$ be a right Hilbert $N$-module and $p \in N$ a projection.  We consider dimensions using the trace $\Tr$ and its restrictions to subalgebras of $N$ and $pNp$.
\begin{enumerate}
\item We have $\dim_{pNp}(Hp) \leq \dim_N(H)$.
\item Let $K \subset Hp$ be a closed $pNp$-submodule. Then $\dim_N(\closure{KN}) = \dim_{pNp}(K)$.
\item Let $P \subset N$ be a von Neumann subalgebra such that $\Tr_{|P}$ is semifinite. Let $K \subset H$ be a closed $P$-submodule. Then $\dim_N(\closure{KN}) \leq \dim_P(K)$.
\end{enumerate}
\end{lemma}

Assume that $\Gamma$ is a countable group and $\Gamma \actson (B,\Tr)$ a trace preserving action on a von Neumann algebra $B$ equipped with a normal semifinite faithful trace $\Tr$. Denote $\cM = B \rtimes \Gamma$ and use the canonical trace $\Tr$ on $\cM$. Let $p \in \cM$ be a projection with $\Tr(p) < \infty$ and $A \subset p\cM p$ a von Neumann subalgebra with $\cN_{p\cM p}(A)\dpr = p\cM p$. Whenever $\Lambda < \Gamma$ is a subgroup, we consider
\begin{equation*}
\cE_\Lambda = \Bigl\{H \Bigm| \text{$H$ is an $A$-$(B \rtimes \Lambda)$-subbimodule of $L^2(p\cM)$ with $\dim_{B \rtimes \Lambda}(H) < \infty$} \Bigr\} \; .
\end{equation*}
If $H \in \cE_\Lambda$, $u \in \cN_{p\cM p}(A)$ and $v \in \cU(B \rtimes \Lambda)$, we have that $u H v$ again belongs to $\cE_\Lambda$. So the closed linear span of all $H \in \cE_\Lambda$ is of the form $L^2(p \cM z(\Lambda))$, where $z(\Lambda)$ is a projection in $\cM \cap (B \rtimes \Lambda)'$. We make $z(\Lambda)$ uniquely determined by requiring that $z(\Lambda)$ is smaller than or equal to the central support of $p$ in $\cM$.

If $\Lambda < \Lambda' < \Gamma$ are subgroups, we have $z(\Lambda) \leq z(\Lambda')$. Indeed, whenever $H \subset L^2(p \cM)$ is an $A$-$(B \rtimes \Lambda)$-subbimodule with $\dim_{B \rtimes \Lambda}(H) < \infty$, we define $K$ as the closed linear span of $H(B \rtimes \Lambda')$. By Lemma \ref{lem.dim}, we get that $\dim_{B \rtimes \Lambda'}(K) < \infty$. Since $H \subset K$ and since this works for all choices of $H$, we conclude that $z(\Lambda) \leq z(\Lambda')$.

The basic construction $\langle \cM, e_{B \rtimes \Lambda} \rangle$ carries a natural semifinite trace $\Tr$ satisfying $\Tr(x e_{B \rtimes \Lambda} x^*) = \Tr(xx^*)$ for all $x \in \cM$.
The projections $e \in A' \cap p \langle \cM,e_{B \rtimes \Lambda} \rangle p$ are precisely the orthogonal projections onto the $A$-$(B \rtimes \Lambda)$-subbimodules $H \subset L^2(p \cM)$. Moreover under this correspondence, we have $\Tr(e) = \dim_{B \rtimes \Lambda}(H)$. We also have the canonical operator valued weight $\cT_\Lambda$ from $\langle \cM, e_{B \rtimes \Lambda} \rangle^+$ to the extended positive part of $\cM$ such that $\Tr = \Tr \circ \cT_\Lambda$. Using the anti-unitary involution $J : L^2(\cM) \recht L^2(\cM) : J (x) = x^*$, we can therefore alternatively define $z(\Lambda)$ as
\begin{align*}
p \, Jz(\Lambda) J & = \bigvee \Bigl\{ e \Bigm| e \in A' \cap p \langle \cM,e_{B \rtimes \Lambda} \rangle p \;\;\text{is a projection with}\;\; \|\cT_\Lambda(e)\| < \infty \Bigr\} \\
& = \bigvee \Bigl\{ \supp(a) \Bigm| a \in A' \cap p \langle \cM,e_{B \rtimes \Lambda} \rangle^+ p \;\;\text{and}\;\; \|\cT_\Lambda(a)\| < \infty \Bigr\} \; .
\end{align*}
If now $\Lambda < \Gamma$ and $\Lambda' < \Gamma$ are subgroups, we can literally repeat the proof of \cite[Proposition 6]{HPV10} and conclude that $z(\Lambda)$ and $z(\Lambda')$ commute with
\begin{equation}\label{eq.intersect}
z(\Lambda \cap \Lambda') = z(\Lambda) \, z(\Lambda') \; .
\end{equation}

We are now ready to prove Theorem \ref{thm.unique-cartan-type-III}.

\begin{proof}[Proof of Theorem \ref{thm.unique-cartan-type-III}]
Denote by $\om : \Gamma \times X \recht \R$ the logarithm of the Radon-Nikodym cocycle. Put $Y = X \times \R$ and equip $Y$ with the measure $m$ given by $dm = d\mu \times \exp(t) dt$, so that the action $\Gamma \actson Y$ given by $g \cdot (x,t) = (g \cdot x, \om(g,x)+t)$ is measure preserving (see \cite{Ma63}). The restricted actions $\Lambda_i \actson (Y,m)$ are still recurrent.

Put $B = L^\infty(Y)$ and denote by $\Tr$ the canonical semifinite trace on $\cM = B \rtimes \Gamma$, given by the infinite invariant measure $m$. Choose a projection $p \in B$ with $0 < \Tr(p) < \infty$. Put $\Sigma_i = \Lambda_i \cap \Sigma$. Since $\Sigma_i$ is a finite group, the von Neumann algebra $p (B \rtimes \Sigma_i) p$ is of type I. Since the action $\Lambda_i \actson (Y,m)$ is recurrent, the von Neumann algebra $p (B \rtimes \Lambda_i) p$ is of type II$_1$. In particular, the inclusion $p (B \rtimes \Sigma_i) p \subset p(B \rtimes \Lambda_i)p$ has no trivial corner in the sense of \cite[Definition 5.1]{HV12} and it follows from \cite[Lemma 5.4]{HV12} that there exists a unitary $u_i \in p(B \rtimes \Lambda_i)p$ such that $E_{p (B \rtimes \Sigma_i) p}(u_i^n) = 0$ for all $n \in \Z - \{0\}$. Since $\Lambda_i \cap \Sigma = \Sigma_i$, we have that $E_{p(B \rtimes \Sigma)p}(x) = E_{p (B \rtimes \Sigma_i) p}(x)$ for all $x \in p(B \rtimes \Lambda_i)p$. So, we get that $E_{p(B \rtimes \Sigma)p}(u_i^n) = 0$ for all $n \in \Z - \{0\}$. We put $v_i = u_i^*$ and have found unitaries $u_i,v_i \in p(B \rtimes \Gamma_i)p$ satisfying
\begin{equation}\label{eq.ui}
E_{p(B \rtimes \Sigma)p}(u_i) = E_{p(B \rtimes \Sigma)p}(v_i) = E_{p(B \rtimes \Sigma)p}(u_i^* v_i) = 0 \; .
\end{equation}

Define the normal trace preserving $*$-homomorphism
$$\Delta : \cM \recht \cM \ovt L(\Gamma) : \Delta(b u_g) = b u_g \ot u_g \quad\text{for all}\;\; b \in B, g \in \Gamma \; .$$
We use the unitaries $u_i$ satisfying \eqref{eq.ui} to prove the following two easy statements.

{\bf Statement 1.} {\it For every $i = 1,2$, we have that $\Delta(p \cM p) \not\prec_{p \cM p \ovt L(\Gamma)} p \cM p \ovt L(\Gamma_i)$.}

{\bf Statement 2.} {\it The von Neumann subalgebra $\Delta(p \cM p) \subset p \cM p \ovt L(\Gamma)$ is not amenable relative to $p \cM p \ovt L(\Sigma)$.}

{\bf Proof of statement 1.} Denote by $|g|$ the length of an element $g \in \Gamma$, i.e.\ the minimal number of factors that is needed to write $g$ as a product of elements in $\Gamma_1$, $\Gamma_2$, with the convention that $|g|=0$ if and only if $g \in \Sigma$. Denote by $Q_m$ the orthogonal projection of $L^2(p \cM p)$ onto the closed linear span of $\{p b u_g p \mid b \in B, g \in \Gamma, |g| \leq m\}$. Denote by $P_m$ the orthogonal projection of $\ell^2(\Gamma)$ onto the closed linear span of $\{u_g \mid g \in \Gamma, |g| \leq m\}$. A direct computation yields
$$(1 \ot P_m)(\Delta(x)) = \Delta(Q_m(x)) \quad\text{for all}\;\; x \in p \cM p \; .$$
Define the unitary $w_n = (u_1 u_2)^n$. Since $Q_m(w_n) = 0$ whenever $n > m/2$, we have that $(1 \ot P_m)(\Delta(w_n)) = 0$ for all $n > m/2$. It follows in particular that for all $g, h \in \Gamma$,
$$E_{p \cM p \ovt L(\Gamma_i)}((1 \ot u_g) \Delta(w_n) (1 \ot u_h)) = 0 \quad\text{whenever}\;\; n > (|g|+|h|+1)/2 \; .$$
So, for every $x,y \in p\cM p \ovt L(\Gamma)$, we get that $\lim_n \|E_{p \cM p \ovt L(\Gamma_i)}(x \Delta(w_n) y)\|_2 = 0$. Hence, $\Delta(p \cM p) \not\prec p \cM p \ovt L(\Gamma_i)$ and statement~1 is proven.

{\bf Proof of statement 2.} Assume that $\Delta(p \cM p)$ is amenable relative $p \cM p \ovt L(\Sigma)$. So we find a positive $\Delta(p\cM p)$-central functional $\Omega$ on $\langle p \cM p \ovt L(\Gamma), e_{p \cM p \ovt L (\Sigma)}\rangle$ such that $\Omega(x) = (\Tr \ot \tau)(x)$ for all $x \in p\cM p \ovt L(\Gamma)$. Note that we can identify
$$\langle p\cM p \ovt L(\Gamma), e_{p\cM p \ovt L(\Sigma)} \rangle = p \cM p \ovt \langle L(\Gamma), e_{L(\Sigma)} \rangle = (p \ot 1) \langle \cM \ovt L(\Gamma), \cM \ovt L(\Sigma)\rangle (p \ot 1) \; .$$
Since $E_{\cM \ovt L(\Sigma)} \circ \Delta = \Delta \circ E_{B \rtimes \Sigma}$ and since the closed linear span of $\Delta(\cM) L^2(\cM \ovt L (\Sigma))$ equals $L^2(\cM \ovt L (\Gamma))$, there is a unique normal unital $*$-homomorphism
$$\Psi : \langle \cM, e_{B \rtimes \Sigma} \rangle \recht \langle \cM \ovt L(\Gamma), e_{\cM \ovt L(\Sigma)} \rangle : \Psi(x e_{B \rtimes \Sigma} y) = \Delta(x) e_{\cM \ovt L(\Sigma)} \Delta(y) \quad\text{for all}\;\; x,y \in \cM \; .$$
The composition of $\Omega$ and $\Psi$ yields a $p\cM p$-central positive functional $\Omega_0$ on $p \langle \cM , e_{B \rtimes \Sigma} \rangle p$ satisfying $\Omega_0(p) = \Tr(p)$. Note that we can view $p \langle \cM , e_{B \rtimes \Sigma} \rangle p$ as the commutant of the right action of $B \rtimes \Sigma$ on $p L^2(\cM)$.

Denote by $H_i \subset p L^2(\cM)$ the closed linear span of all $p b u_g$ with $b \in B$ and $g \in \Gamma$ such that a reduced expression of $\Gamma$ as an alternating product of elements in $\Gamma_1 - \Sigma$ and $\Gamma_2 - \Sigma$ starts with a factor in $\Gamma_i - \Sigma$. Denote $H_0 = p L^2(B)$. So we have the orthogonal decomposition $p L^2(\cM) = H_0 \oplus H_1 \oplus H_2$. Denote by $e_i : p L^2(\cM) \recht H_i$ the orthogonal projection. Note that $e_i$ is a projection in $p \langle \cM , e_{B \rtimes \Sigma} \rangle p$. By \eqref{eq.ui}, the projections $u_2 (e_0 + e_1) u_2^*$ and $v_2 (e_0 + e_1) v_2^*$ are orthogonal and lie under $e_2$. Since $\Omega_0$ is $p\cM p$-central, it follows that
$$2 \Omega_0(e_0 + e_1) \leq \Omega_0(e_2) \; .$$
It similarly follows that $2 \Omega_0(e_2) \leq \Omega_0(e_1)$. Together, it follows that $\Omega_0(e_0 + e_1) = \Omega_0(e_2) = 0$. Since $e_0 + e_1 + e_2 = p$, we obtain the contradiction that $\Omega_0(p) = 0$. So also statement~2 is proven.

Assume now that $L^\infty(X) \rtimes \Gamma$ admits a Cartan subalgebra that is not unitarily conjugate to $L^\infty(X)$. The first paragraphs of the proof of \cite[Theorem A]{HV12} are entirely general and yield an abelian von Neumann subalgebra $A \subset p\cM p$ such that $\cN_{p\cM p}(A)\dpr = p \cM p$ and $A \not\prec Bq$ whenever $q \in B$ is a projection with $\Tr(q) < \infty$. So to prove the theorem, we fix an abelian von Neumann subalgebra $A \subset p \cM p$ with $\cN_{p\cM p}(A)\dpr = p \cM p$. We have to find a projection $q \in B$ with $\Tr(q) < \infty$ and $A \prec Bq$.

Note that $\Delta(A) \subset p\cM p \ovt L(\Gamma)$ is an abelian, hence amenable, von Neumann subalgebra whose normalizer contains $\Delta(p \cM p)$. We view $p \cM p \ovt L(\Gamma)$ as the amalgamated free product of $p \cM p \ovt L(\Gamma_1)$ and $p \cM p \ovt L(\Gamma_2)$ over their common von Neumann subalgebra $p \cM p \ovt L(\Sigma)$. A combination of Theorem \ref{thm.AFP} and statements 1 and 2 above implies that $\Delta(A) \prec p\cM p \ovt L(\Sigma)$. So there is no sequence of unitaries $(w_n)$ in $\cU(A)$ satisfying $\lim_n \|E_{p \cM p \ovt L(\Sigma)}(x \Delta(w_n) y)\|_2 = 0$ for all $x, y \in p\cM p \ovt L(\Gamma)$. This means that we can find $\eps > 0$ and $h_1,\ldots,h_m \in \Gamma$ such that
\begin{equation}\label{eq.ineq}
\sum_{i,j = 1}^m \|E_{p\cM p \ovt L(\Sigma)}((1 \ot u_{h_i}^*)\Delta(a)(1 \ot u_{h_j}))\|_2^2 \geq \eps \quad\text{for all}\;\; a \in \cU(A) \; .
\end{equation}
Define the positive element $T = \sum_{i=1}^m p u_{h_i} e_{B \rtimes \Sigma} u_{h_i}^* p$ in $p \langle \cM , e_{B \rtimes \Sigma} \rangle p$. The left hand side of \eqref{eq.ineq} equals $\Tr(T a T a^*)$. Denote by $S$ the element of smallest $\|\,\cdot\,\|_{2,\Tr}$-norm in the weakly closed convex hull of $\{a T a^* \mid a \in \cU(A)\}$. Then $S$ is a nonzero element of $A' \cap p \langle \cM,e_{B \rtimes \Sigma} \rangle p$ and $\Tr(S) < \infty$. Using the notation introduced before this proof, this means that $z(\Sigma) \neq 0$.

Since the action $\Gamma \actson Y$ is free, we have $\cM \cap B' = B$. So the projections $z(\Sigma)$ and $z(\Gamma_i)$ belong to $B$ and are, respectively, $\Sigma$- and $\Gamma_i$-invariant. We prove below that $z(\Sigma)$ is a $\Gamma$-invariant projection in $B$. We prove this by showing that $z(\Gamma_1) = z(\Sigma) = z(\Gamma_2)$.

Since $\Sigma < \Gamma_i$, we have that $z(\Sigma) \leq z(\Gamma_i)$ for every $i = 1,2$. We claim that the equality holds. Assume that $z(\Sigma) < z(\Gamma_1)$. Note that both projections belong to $B$. Choose a nonzero projection $q \in B$ with $\Tr(q) < \infty$ and $q \leq z(\Gamma_1) - z(\Sigma)$. Choose $H \in \cE_{\Gamma_1}$ such that $H q \neq \{0\}$. By Lemma \ref{lem.dim}, we have
$$\dim_{q(B \rtimes \Gamma_1)q}(H q) \leq \dim_{B \rtimes \Gamma_1}(H) < \infty \; .$$
We conclude that $L^2(p \cM q)$ admits a nonzero $A$-$q(B \rtimes \Gamma_1)q$-subbimodule $K$ that is finitely generated as a right Hilbert module. Since $q \perp z(\Sigma)$, we also know that $L^2(p \cM q)$ does not admit an $A$-$q(B \rtimes \Sigma)q$-subbimodule that is finitely generated as a right Hilbert module. We then encode $K$ as an integer $n$, a projection $q_1 \in M_n(\C) \ot q(B \rtimes \Gamma_1)q$, a nonzero partial isometry $V \in p(M_{1,n}(\C) \ot \cM)q_1$ and a normal unital $*$-homomorphism $\vphi : A \recht q_1(M_n(\C) \ot (B \rtimes \Gamma_1))q_1$ such that
\begin{equation}\label{eq.wehave}
a V = V \vphi(a) \;\;\text{for all}\;\; a \in A \quad\text{and}\quad \vphi(A) \not\prec_{M_n(\C) \ot q(B \rtimes \Gamma_1)q} q(B \rtimes \Sigma)q \; .
\end{equation}
Let $u \in \cN_{p\cM p}(A)$ and write $u a u^* = \al(a)$ for all $a \in A$. Then $V^* u V$ is an element of $q_1(M_n(\C) \ot \cM) q_1$ satisfying
$$V^* u V \; \vphi(a) = \vphi(\alpha(a)) \; V^* u V \quad\text{for all}\;\; a \in A \; .$$
By \eqref{eq.wehave} and \cite[Theorem 2.4]{CH08}, it follows that $V^* u V \in q_1(M_n(\C) \ot (B \rtimes \Gamma_1))q_1$. This holds for all $u \in \cN_{p\cM p}(A)$. Since the linear span of $\cN_{p\cM p}(A)$ is strongly dense in $p \cM p$, and writing $q_2 = V^* V$, we have found a nonzero projection $q_2 \in M_n(\C) \ot (B \rtimes \Gamma_1)$ with the property that
$$q_2(M_n(\C) \ot \cM)q_2 = q_2(M_n(\C) \ot (B \rtimes \Gamma_1))q_2 \; .$$
In the von Neumann algebra $M_n(\C) \ot (B \rtimes \Gamma_1)$, the projection $q_2$ is equivalent with a projection in $D_n(\C) \ot B$, where $D_n(\C) \subset M_n(\C)$ is the diagonal subalgebra. So, we find a nonzero projection $q_3 \in B$ satisfying $q_3 \cM q_3 = q_3 (B \rtimes \Gamma_1) q_3$. As in \eqref{eq.ui}, there however exists a unitary $v \in q_3(B \rtimes \Gamma_2)q_3$ with the property that $E_{q_3(B \rtimes \Sigma)q_3}(v) = 0$. It follows that $v$ belongs to $q_3 \cM q_3$, but is orthogonal to $q_3 (B \rtimes \Gamma_1) q_3$. We have reached a contradiction and conclude that $z(\Sigma) = z(\Gamma_1)$. By symmetry, we also have that $z(\Sigma) = z(\Gamma_2)$.

Since $z(\Gamma_i)$ is a $\Gamma_i$-invariant projection in $B$, we conclude that $z(\Sigma)$ is a nonzero $\Gamma$-invariant projection in $B$.
Take now $g_1,\ldots,g_n \in \Gamma$ such that $\Sigma_0 = \bigcap_{k=1}^n g_k \Sigma g_k^{-1}$ is finite. By definition, we have $z(g_k \Sigma g_k^{-1}) = \sigma_{g_k}(z(\Sigma))$. Since $z(\Sigma)$ is $\Gamma$-invariant, it follows that $z(g_k \Sigma g_k^{-1}) = z(\Sigma)$ for every $k$. Using \eqref{eq.intersect}, we conclude that $z(\Sigma) = z(\Sigma_0)$. In particular, $z(\Sigma_0) \neq 0$. So we find a nonzero $A$-$(B \rtimes \Sigma_0)$-subbimodule $H$ of $L^2(p \cM)$ with $\dim_{B \rtimes \Sigma_0}(H) < \infty$. A fortiori, $H$ is an $A$-$B$-bimodule. Since $\Sigma_0$ is finite, also $\dim_B(H) < \infty$. Taking a projection $q \in B$ with $\Tr(q) < \infty$ and $H q \neq \{0\}$, it follows from Lemma \ref{lem.dim} that we have found a nonzero $A$-$Bq$-subbimodule of $L^2(p\cM q)$ having finite right dimension. This precisely means that $A \prec Bq$ and hence, ends the proof of the theorem.
\end{proof}

We can now deduce Proposition \ref{prop.Wstar-superrigid}.

\begin{proof}[Proof of Proposition \ref{prop.Wstar-superrigid}]
Write $X = \R^5/\R_+ \times [0,1]^\Gamma$ and $Y = \R^5 \times [0,1]^\Gamma$. Put $G = \Gamma \times \R_+$ and consider the action $G \actson Y$ given by
$$(g,\alpha) \cdot (x,y) = (\alpha \pi(g) \cdot x, g \cdot y) \quad\text{for all}\;\; g \in \Gamma, \al \in \R_+, x \in \R^5, y \in [0,1]^\Gamma \; .$$
Note that the restricted action $\Gamma \actson Y$ is infinite measure preserving and can be identified with the Maharam extension of $\Gamma \actson X$. Since the Bernoulli action $\Gamma \actson [0,1]^\Gamma$ is mixing, we use throughout the proof that the restriction of $\Gamma \actson Y$ to a subgroup $\Lambda < \Gamma$ is ergodic whenever $\pi(\Lambda)$ acts ergodically on $\R^5$ (see e.g.\ \cite[Proposition 2.3]{Sc82}). Using \cite[Lemma 5.6]{PV08}, we find in particular that $\Gamma \actson Y$ is ergodic, meaning that $\Gamma \actson X$ is of type III$_1$.

Let $\cG$ be a Polish group in Popa's class $\cU_\text{fin}$, i.e.\ a closed subgroup of the unitary group of a II$_1$ factor, e.g.\ any countable group.
We claim that every measurable $1$-cocycle $\om : G \times Y \recht \cG$ is cohomologous to a continuous group homomorphism $G \recht \cG$. As explained in detail in \cite[Step 1 of the proof of Theorem 21]{KS12}, it follows from \cite[Theorem 5.3]{PV08} that up to cohomology, we may assume that the restriction of $\om$ to $\SL(5,\Z)$ is a group homomorphism. By \cite[Lemma 5.6]{PV08}, the diagonal action $\SL(3,\Z) \actson \R^3 \times \R^3$ is ergodic. It follows that the diagonal action $\Sigma \actson \R^5 \times \R^5$ is ergodic as well. But then also the diagonal action $\Sigma \actson Y \times Y$ is ergodic. Since the restriction of $\om$ to $\Sigma$ is a homomorphism and since $\Sigma$ commutes with the natural copies of $\Z$ and $\R_+$ inside $G$, it now follows from \cite[Lemma 5.5]{PV08} that $\om$ is also a homomorphism on $\Z$ and on $\R_+$. Because $\SL(5,\Z)$, $\Z$ and $\R_+$ together generate $G$, we have proven the claim that $\om$ is cohomologous to a group homomorphism.

We next prove that $\R_+$ is the only open normal subgroup of $G$ that acts properly on $Y$. Indeed, if $G_0$ is such a subgroup, we first have that $\R_+ \subset G_0$ because $\R_+$ is connected. So $G_0 = \Gamma_0 \times \R_+$ where $\Gamma_0$ is a normal subgroup of $\Gamma$ that acts properly on $Y$. Then $\pi(\Gamma_0)$ is a normal subgroup of $\SL(5,\Z)$. So either $\pi(\Gamma_0) = \{1\}$ or $\pi(\Gamma_0)$ has finite index in $\SL(5,\Z)$. In the latter case, $\Gamma_0$ acts ergodically on $Y$, rather than properly. In the former case, $\Gamma_0$ only acts by the Bernoulli shift and the properness forces $\Gamma_0$ to be finite. But $\Gamma$ is an icc group, so that $\Gamma_0 = \{e\}$.

The cocycle superrigidity of $G \actson Y$, together with the previous paragraph and \cite[Lemma 5.10]{PV08}, now implies that the only actions that are stably orbit equivalent with $\Gamma \actson X$ are the induced $\Gamma'$-actions, given an embedding of $\Gamma$ into $\Gamma'$.

So to conclude the proof, it remains to show that $L^\infty(X) \rtimes \Gamma$ has a unique Cartan subalgebra up to unitary conjugacy. This follows from Theorem \ref{thm.unique-cartan-type-III}, using the subgroups $\SL(2,\Z) < \SL(5,\Z)$ (embedded in the upper left corner) and $\Z < \Sigma \times \Z$ that act recurrently on $X$ and intersect $\Sigma$ trivially.
\end{proof}

\begin{remark}\label{rem.more-general}
In the formulation of Theorem \ref{thm.unique-cartan-type-III}, we required the existence of subgroups $\Lambda_i < \Gamma_i$ that intersect $\Sigma$ finitely and that act in a recurrent way on $(X,\mu)$. It is actually sufficient to impose the following more ergodic theoretic condition. Denote by $\Gamma \actson (Y,m)$ the (infinite measure preserving) Maharam extension of $\Gamma \actson (X,\mu)$. Consider the orbit equivalence relations $\cR(\Gamma_i \actson Y)$ and $\cR(\Sigma \actson Y)$, as well as their restrictions to nonnegligible subsets of $Y$. It is then sufficient to assume that for every Borel set $\cU \subset Y$ with $0 < m(\cU) < \infty$, almost every $\cR(\Gamma_i \actson Y)_{|\cU}$-equivalence class consists of infinitely many $\cR(\Sigma \actson Y)_{|\cU}$-equivalence classes. Indeed, writing $B = L^\infty(Y)$, it then follows from \cite[Lemma 2.6]{IKT08} that for every projection $p \in B$ with $0 < \Tr(B) < \infty$, there exist unitaries $u_i,v_i \in p(B \rtimes \Gamma_i)p$ satisfying \eqref{eq.ui}. So the proof of Theorem \ref{thm.unique-cartan-type-III} goes through.
\end{remark}

\end{document}